\documentclass[oneside,english]{amsart}
\usepackage[T1]{fontenc}
\usepackage[latin9]{inputenc}
\usepackage{amssymb}
\usepackage{esint}

\makeatletter
\numberwithin{equation}{section} 
\numberwithin{figure}{section} 
  \theoremstyle{plain}
  \newtheorem{thm}{Theorem}[section]
  \theoremstyle{remark}
  \newtheorem*{acknowledgement*}{Acknowledgement}

\usepackage{babel}
\makeatother

\begin{document}

\title{Sums of Zeros for Certain Special Functions}

\author{Ruiming Zhang}

\begin{abstract}
In this work we evaluates sums of the zeros for the Bessel function
$J_{\nu}(z)$, the Airy function $A(z)$, the Riemann zeta function
$\zeta(z)$, $L$-series $L(s,\chi)$ with real primitive characters,
Ramanujan's entire function (a.k.a. $q$-Airy function) $A_{q}(z)$,
$q$-Bessel function $J_{\nu}^{(2)}(z;q)$. 
\end{abstract}


\curraddr{School of Mathematical Sciences\\
Guangxi Normal University\\
Guilin City, Guangxi 541004\\
P. R. China.}


\email{ruimingzhang@yahoo.com}

\maketitle

\section{Introduction}

Given an entire function $f(z)$, it is interesting to find formulas
for various sums of zeros for $f(z)$. Generally speaking, it is relatively
easy to calculate the multiple sums and hard to find the related power
sums. In this work we prove an identity for certain class of entire
functions that is very similar to the formulas for polynomials. This
class includes the Bessel function $J_{\nu}(z)$, the Riemann zeta
function $\zeta(z)$, $L$-series $L(s,\chi)$ with real primitive
characters, Ramanujan's entire function (a.k.a. $q$-Airy function)
$A_{q}(z)$, $q$-Bessel function $J_{\nu}^{(2)}(z;q)$. Using this
identity we give closed form evaluations of certain multiple sums
and power sums of the zeros of these functions. 

The work is divided into four sections, in section \ref{sec:Preliminaries}
we present some facts on the special functions; We state and derive
the identity in section \ref{sec:Main-Results}; In the section \ref{sec:Applications}
we apply the identity to some special functions.

\section{Preliminaries\label{sec:Preliminaries}}

\subsection{Bessel type functions \cite{Abramowitz,Andrews2,Erdelyi,Szego,Watson}}

For any complex number $z\in\mathbb{C}$, the Euler's $\Gamma(z)$
is defined as \begin{align}
\frac{1}{\Gamma(z)}: & =z\prod_{j=1}^{\infty}\left(1+\frac{z}{j}\right)\left(1+\frac{1}{j}\right)^{-z}.\label{eq:1}\end{align}
The Bessel functions $J_{\nu}(z)$ is defined as\begin{align}
J_{\nu}(z): & =\sum_{n=0}^{\infty}\frac{(-1)^{n}}{\Gamma(n+\nu+1)n!}\left(\frac{z}{2}\right)^{\nu+2n}.\label{eq:2}\end{align}
It is well known that, for $\nu>-1$, the zeros of the even entire
function $z^{-\nu}J_{\nu}(z)$ are real and simple. We denote the
positive zeros as\begin{align}
0 & <j_{\nu,1}<j_{\nu,2}<\dots<j_{\nu,n},\dots.\label{eq:3}\end{align}
They satisfy the following identity \begin{align}
\sum_{n=1}^{\infty}\frac{1}{j_{\nu,n}^{2}} & =\frac{1}{4(\nu+1)}.\label{eq:4}\end{align}
The Bessel functions have the infinite product expansion \begin{align}
J_{\nu}(z) & =\frac{z^{\nu}}{2^{\nu}\Gamma(\nu+1)}\prod_{n=1}^{\infty}\left(1-\frac{z^{2}}{j_{\nu,n}^{2}}\right).\label{eq:5}\end{align}
The Bessel function $J_{\nu}(z)$ also satisfies the following second
order differential equation\begin{align}
z^{2}\frac{d^{2}y(z)}{dz^{2}}+z\frac{dy(z)}{dz}+(z^{2}-\nu^{2})y(z) & =0.\label{eq:6}\end{align}
The Airy function is defined as \cite{Szego} \begin{align}
A(z): & =\frac{\pi}{3}\sum_{n=0}^{\infty}\frac{1}{\Gamma(n+2/3)n!}\left(-\frac{z}{3}\right)^{3n}+\frac{\pi z}{9}\sum_{n=0}^{\infty}\frac{1}{\Gamma(n+4/3)n!}\left(-\frac{z}{3}\right)^{3n}.\label{eq:7}\end{align}
Clearly \cite{Abramowitz,Vallee}, \begin{align}
A(z) & =\frac{\pi}{\sqrt[3]{3}}\mathrm{Ai}\left(-\frac{z}{\sqrt[3]{3}}\right).\label{eq:8}\end{align}
The Airy function $A(z)$ satisfies the following second order differential
equation\begin{align}
y''(z)+\frac{z}{3}y(z) & =0.\label{eq:9}\end{align}
It is known that $A(z)$ has infinitely many real zeros, all of them
are positive and simple. Let us denote them as\begin{align}
0 & <i_{1}<i_{2}<\dots.\label{eq:10}\end{align}
It is also known that\begin{align}
i_{\nu} & \sim\nu^{2/3}\label{eq:11}\end{align}
as $\nu\to\infty$. 

From the infinite product expansion for $\mathrm{Ai}(z)$ \cite{Vallee}
we get\begin{align}
A(z) & =\frac{\pi}{3}\frac{e^{\kappa z}}{\Gamma(2/3)}\prod_{n=1}^{\infty}\left(1-\frac{z}{i_{n}}\right)e^{z/i_{n}},\label{eq:12}\end{align}
where\begin{align}
\kappa^{2}: & =\sum_{n=1}^{\infty}\frac{1}{i_{n}^{2}}=\frac{3\Gamma(2/3)^{4}}{4\pi^{2}},\quad k>0.\label{eq:13}\end{align}
Thus,\begin{align}
\frac{9\Gamma(2/3)^{2}A(\sqrt{z})A(-\sqrt{z})}{\pi^{2}} & =\prod_{n=1}^{\infty}\left(1-\frac{z}{i_{n}^{2}}\right).\label{eq:14}\end{align}
From the product formula \begin{align}
\mathrm{Ai}(z)\mathrm{Ai}(-z) & =\frac{1}{\sqrt[3]{2}\pi}\int_{-\infty}^{\infty}\mathrm{Ai}(\sqrt[3]{4}t^{2})e^{2izt}dt\label{eq:15}\end{align}
we obtain\begin{align}
\prod_{n=1}^{\infty}\left(1-\frac{z}{i_{n}^{2}}\right) & =2^{5/3}3^{4/3}\Gamma\left(\frac{2}{3}\right)^{2}\int_{0}^{\infty}\mathrm{Ai}(2^{2/3}t^{2})\cos\left(\frac{2tz^{1/2}}{\sqrt[3]{3}}\right)dt\label{eq:16}\\
 & =2^{5/3}3^{4/3}\Gamma\left(\frac{2}{3}\right)^{2}\sum_{n=0}^{\infty}\frac{(-1)^{n}m_{n}}{(2n)!}\left(\frac{2}{\sqrt[3]{3}}\right)^{2n}z^{n},\nonumber \end{align}
where\begin{align}
m_{n}: & =\int_{0}^{\infty}\mathrm{Ai}(2^{2/3}t^{2})t^{2n}dt,\quad n\in\mathbb{N}\cup\left\{ 0\right\} .\label{eq:17}\end{align}
From \begin{align}
\mathrm{Ai}\left(\left(\frac{3x}{2}\right)^{2/3}\right) & =\frac{1}{\pi\sqrt{3}}\left(\frac{3x}{2}\right)^{2/3}K_{1/3}(x)\label{eq:18}\end{align}
we obtain\begin{align}
\mathrm{Ai}(z) & =\frac{z^{1/2}}{\pi\sqrt{3}}K_{1/3}\left(\frac{2z^{3/2}}{3}\right).\label{eq:19}\end{align}
Hence,\begin{align}
m_{n} & =\frac{3^{(2n-1)/3-1/2}}{\pi2^{(4n)/3+1}}\int_{0}^{\infty}K_{1/3}(u)u^{(2n-1)/3}du\label{eq:20}\\
 & =\frac{3^{(2n-1)/3-1/2}}{\pi2^{(2n+7)/3}}\Gamma\left(\frac{n}{3}+\frac{1}{6}\right)\Gamma\left(\frac{n}{3}+\frac{1}{2}\right).\nonumber \end{align}
Then,\begin{align}
\prod_{n=1}^{\infty}\left(1-\frac{z}{i_{n}^{2}}\right) & =\sum_{n=0}^{\infty}(-1)^{n}\alpha_{n}z^{n},\label{eq:21}\end{align}
where\begin{align}
\alpha_{n}: & =\frac{\sqrt{3}\Gamma\left(\frac{2}{3}\right)^{2}}{\sqrt[3]{4}\pi}\frac{16^{n/3}\Gamma\left(\frac{n}{3}+\frac{1}{6}\right)\Gamma\left(\frac{n}{3}+\frac{1}{2}\right)}{(2n)!}.\label{eq:22}\end{align}

\subsection{$q$-Series \cite{Andrews2,Gasper}}

Assume that $0<q<1$, let\begin{align}
(z;q)_{\infty}: & =\prod_{n=0}^{\infty}(1-zq^{n}),\label{eq:23}\end{align}
 \begin{align}
(z;q)_{n}: & =\frac{(z;q)_{\infty}}{(zq^{n};q)_{\infty}},\label{eq:24}\end{align}
and\begin{align}
(z_{1},z_{2},\dots,z_{m};q)_{n}: & =\prod_{j=1}^{m}(z_{j};q)_{n}\label{eq:25}\end{align}
for any $m\in\mathbb{N}$,$n\in\mathbb{Z}$ and $z,z_{1},z_{2},\dots,z_{m}\in\mathbb{C}$.

The $q$-Bessel functions $J_{\nu}^{(2)}(z;q)$ is defined as

\begin{align}
J_{\nu}^{(2)}(z;q): & =\frac{(q^{\nu+1};q)_{\infty}}{(q;q)_{\infty}}\sum_{n=0}^{\infty}\frac{\left(-q^{n+\nu}\right)^{n}}{(q,q^{\nu+1};q)_{n}}\left(\frac{z}{2}\right)^{\nu+2n}.\label{eq:26}\end{align}
It is known that for $\nu>-1$, all the zeros of $z^{-\nu}J_{\nu}^{(2)}(z;q)$
are real and simple. We denote the positive zeros as\begin{align}
0 & <j_{\nu,1}(q)<j_{\nu,2}(q)<\dots<j_{\nu,n}(q)<\dots.\label{eq:27}\end{align}
It is known that \cite{Hayman,Ismail} \begin{align}
\sum_{n=1}^{\infty}\frac{1}{j_{\nu,n}(q)} & <\infty.\label{eq:28}\end{align}
Then,\begin{align}
\left(\frac{z}{2}\right)^{-\nu}J_{\nu}^{(2)}(z;q) & =\frac{(q^{\nu+1};q)_{\infty}}{(q;q)_{\infty}}\prod_{n=1}^{\infty}\left(1-\frac{z^{2}}{j_{\nu,n}^{2}(q)}\right).\label{eq:29}\end{align}
The Ramanujan's entire function (a.k.a. $q$-Airy function) $A_{q}(z)$
is defined as \cite{Andrews2,Ismail} \begin{align}
A_{q}(z): & =\sum_{n=0}^{\infty}\frac{q^{n^{2}}(-z)^{n}}{(q;q)_{n}}.\label{eq:30}\end{align}
It is known that all the zeros of $A_{q}(z)$ are positive and simple,
we let them be\begin{align}
0 & <i_{1}(q)<i_{2}(q)<\dots<i_{n}(q)<\dots.\label{eq:31}\end{align}
It is also known that \begin{align}
\sum_{n=1}^{\infty}\frac{1}{i_{n}(q)} & <\infty.\label{eq:32}\end{align}
Hence,\begin{align}
A_{q}(z) & =\prod_{n=1}^{\infty}\left(1-\frac{z}{i_{n}(q)}\right).\label{eq:33}\end{align}

\subsection{The Riemann Zeta Function $\zeta(s)$ \cite{Abramowitz,Andrews2,Davenport,Hua,Titchmarsh}}

The Riemann zeta function $\zeta(s)$ is defined as\begin{align}
\zeta(s): & =\sum_{n=1}^{\infty}\frac{1}{n^{s}},\quad\Re(s)>1.\label{eq:34}\end{align}
The entire functions \begin{align}
\xi(z): & =\frac{z(z-1)}{2\pi^{z/2}}\Gamma\left(\frac{z}{2}\right)\zeta(z),\label{eq:35}\end{align}
and \begin{align}
\Xi(z): & =\xi\left(\frac{1}{2}+iz\right)\label{eq:36}\end{align}
are of order $1$. They satisfy the following functional equations,\begin{align}
\xi(z) & =\xi(1-z),\label{eq:37}\end{align}
and \begin{align}
\Xi(z) & =\Xi(-z).\label{eq:38}\end{align}
The function $\Xi(z)$ has an integral representation\begin{align}
\Xi(z) & =\int_{0}^{\infty}\phi(t)\cos(zt)dt,\label{eq:39}\end{align}
where\begin{align}
\phi(t): & =4\pi\sum_{n=1}^{\infty}\left\{ 2\pi n^{4}e^{-9t/2}-3n^{2}e^{-5t/2}\right\} \exp\left(-n^{2}\pi e^{-2t}\right).\label{eq:40}\end{align}
Evidently,\begin{align}
\Xi(z) & =\sum_{n=0}^{\infty}\frac{(-1)^{n}b_{n}}{(2n)!}z^{2n},\label{eq:41}\end{align}
where\begin{align}
b_{n}: & =\int_{0}^{\infty}t^{2n}\phi(t)dt.\label{eq:42}\end{align}
It is well known that $\phi(t)$ is positive, even and fast decreasing
on $\mathbb{R}$. From formula \eqref{eq:42}, it is clear that \begin{align}
b_{0} & =\Xi(0)>0,\label{eq:43}\end{align}
we list all its zeros with positive real part first according to their
real parts, then their imaginary parts, \begin{equation}
z_{1},z_{2},\dots,z_{n},\dots.\label{eq:44}\end{equation}
Then,\begin{align}
\Xi(z) & =b_{0}\prod_{n=1}^{\infty}\left(1-\frac{z^{2}}{z_{n}^{2}}\right).\label{eq:45}\end{align}
Thus,\begin{align}
\prod_{n=1}^{\infty}\left(1-\frac{z}{z_{n}^{2}}\right) & =\frac{\Xi(\sqrt{z})}{b_{0}}=\sum_{n=0}^{\infty}(-1)^{n}\beta_{n}z^{n},\label{eq:46}\end{align}
where\begin{align}
\beta_{0}:=1,\quad\beta_{n}: & =\frac{b_{n}}{(2n)!b_{0}},\quad n\in\mathbb{N}.\label{eq:47}\end{align}
More generally, let $\chi(n)$ be a real primitive character to modulus
$m$. The function $L(s,\chi)$ is defined as \begin{align}
L(s,\chi): & =\sum_{n=1}^{\infty}\frac{\chi(n)}{n^{s}},\quad\Re(s)>1.\label{eq:48}\end{align}
Let\begin{align}
a: & =\begin{cases}
0, & \chi(-1)=1\\
1, & \chi(-1)=-1\end{cases},\label{eq:49}\end{align}
it is known that for\begin{align}
\xi(s|\chi,a): & =\left(\frac{\pi}{m}\right)^{-(s+a)/2}\Gamma\left(\frac{s+a}{2}\right)L(s,\chi),\label{eq:50}\end{align}
then,\begin{align}
\xi(1-s|\chi,a) & =\xi(s|\chi,a).\label{eq:51}\end{align}
It is also known that the entire function\begin{align}
\Xi(z|\chi,a): & =\xi\left(\frac{1}{2}+iz|\chi,a\right)\label{eq:52}\end{align}
is of order $1$, even and has an integral representation\begin{align}
\Xi(z|\chi,a) & =\frac{1}{2}\int_{-\infty}^{\infty}e^{-izt}\phi(t|\chi,a)dt,\label{eq:53}\end{align}
where\begin{align}
\phi(t|\chi,0): & =2e^{-t/2}\sum_{n=-\infty}^{\infty}\chi(n)e^{-n^{2}\pi e^{-2t}/m},\label{eq:54}\end{align}
and\begin{align}
\phi(t|\chi,1): & =2e^{-3t/2}\sum_{n=-\infty}^{\infty}n\chi(n)e^{-n^{2}\pi e^{-2t}/m}.\label{eq:55}\end{align}
From $a=0$, \begin{align}
\sum_{n=-\infty}^{\infty}\chi(n)e^{-n^{2}\pi x/m} & =\frac{1}{\sqrt{x}}\sum_{n=-\infty}^{\infty}\chi(n)e^{-n^{2}\pi/(mx)},\quad x>0,\label{eq:56}\end{align}
and $a=1$,\begin{align}
\sum_{n=-\infty}^{\infty}n\chi(n)e^{-n^{2}\pi x/m} & =x^{-3/2}\sum_{n=-\infty}^{\infty}n\chi(n)e^{-n^{2}\pi/(mx)},\quad x>0.\label{eq:57}\end{align}
It is easy to verify that\begin{align}
\phi(-t|\chi,a) & =\phi(t|\chi,a),\quad t\in\mathbb{R}.\label{eq:58}\end{align}
Then, \begin{align}
\Xi(z|\chi,a) & =\int_{0}^{\infty}\phi(t|\chi,a)\cos(zt)dt\label{eq:59}\\
 & =\sum_{n=0}^{\infty}\frac{(-1)^{n}b_{n}(\chi,a)z^{2n}}{(2n)!},\nonumber \end{align}
where\begin{align}
b_{n}(\chi,a): & =\int_{0}^{\infty}t^{2n}\phi(t|\chi,a)dt.\label{eq:60}\end{align}
Assume that\begin{equation}
b_{0}(\chi,a)=\Xi(0|\chi,a)=\xi\left(\frac{1}{2}|\chi,a\right)\neq0,\label{eq:61}\end{equation}
then formula \eqref{eq:59} says that $0$ is not a zero of $\Xi(z|\chi,a)$.
We list all its zeros with positive real part first according to their
real parts, then their imaginary parts, \begin{equation}
z_{1}(\chi,a),z_{2}(\chi,a),\dots,z_{n}(\chi,a),\dots.\label{eq:62}\end{equation}
Then,\begin{align}
\Xi(z|\chi,a) & =b_{0}(\chi,a)\prod_{n=1}^{\infty}\left(1-\frac{z^{2}}{z_{n}(\chi,a)^{2}}\right).\label{eq:63}\end{align}
Thus,\begin{align}
\prod_{n=1}^{\infty}\left(1-\frac{z}{z_{n}(\chi,a)^{2}}\right) & =\frac{\Xi(\sqrt{z}|\chi,a)}{b_{0}(\chi,a)}=\sum_{n=0}^{\infty}(-1)^{n}\beta_{n}(\chi,a)z^{n},\label{eq:64}\end{align}
where\begin{align}
\beta_{0}(\chi,a): & =1,\quad\beta_{n}(\chi,a):=\frac{b_{n}(\chi,a)}{(2n)!b_{0}(\chi,a)},\quad n\in\mathbb{N}.\label{eq:65}\end{align}

\section{Main Results\label{sec:Main-Results}}

\begin{thm}
\label{thm:3.1}Given a sequence of non-zero complex numbers $\left\{ \lambda_{n}\right\} _{n=1}^{\infty}\subset\mathbb{C}$
satisfying \begin{equation}
\sum_{n=1}^{\infty}\left|\lambda_{n}\right|<\infty.\label{eq:66}\end{equation}
Let\begin{align}
f(z): & =\prod_{n=1}^{\infty}(1-z\lambda_{n})=\sum_{n=0}^{\infty}(-1)^{n}\sigma_{n}z^{n},\label{eq:67}\end{align}
then\begin{align}
\frac{(-1)^{n}f^{(n)}(z)}{n!f(z)} & =\sum_{1\le k_{1}<k_{2}\dots<k_{n}}\frac{\lambda_{k_{1}}\lambda_{k_{2}}\dotsb\lambda_{k_{n}}}{(1-z\lambda_{k_{1}})(1-z\lambda_{k_{2}})\dotsb(1-z\lambda_{k_{n}})}\label{eq:68}\end{align}
for all $n\in\mathbb{N}$ and all $z\in\mathbb{C}$ which is not in
the sequence $\left\{ \lambda_{n}^{-1}\right\} _{n=1}^{\infty}$.
In particular, we have\begin{align}
\sigma_{0} & =1,\quad\sigma_{1}=\sum_{k=1}^{\infty}\lambda_{k},\label{eq:69}\end{align}
 and\begin{align}
\sigma_{n} & =\frac{(-1)^{n}f^{(n)}(0)}{n!}=\sum_{1\le k_{1}<k_{2}\dots<k_{n}}\lambda_{k_{1}}\lambda_{k_{2}}\dotsb\lambda_{k_{n}}.\label{eq:70}\end{align}
Let \begin{align}
s_{n}: & =\sum_{k=1}^{\infty}\lambda_{k}^{n},\quad n\in\mathbb{N},\label{eq:71}\end{align}
then for $n\in\mathbb{N}$ we have \begin{align}
s_{n} & =(-1)^{n-1}n\sigma_{n}+\sum_{j=1}^{n-1}(-1)^{j-1}\sigma_{j}s_{n-j}\label{eq:72}\end{align}
and \begin{align}
\frac{s_{n}}{c^{n}} & =\det\left(\begin{array}{cccccc}
1 & 0 & 0 & \dots & 0 & \frac{\sigma_{1}}{c}\\
-\frac{\sigma_{1}}{c} & 1 & 0 & \dots & 0 & -\frac{2\sigma_{2}}{c^{2}}\\
\frac{\sigma_{2}}{c^{2}} & -\frac{\sigma_{1}}{c} & 1 & \dots & 0 & \frac{3\sigma_{3}}{c^{3}}\\
\vdots & \vdots & \vdots & \ddots & \vdots & \vdots\\
\frac{(-1)^{n-2}\sigma_{n-2}}{c^{n-2}} & \frac{(-1)^{n-3}\sigma_{n-3}}{c^{n-3}} & \frac{(-1)^{n-4}\sigma_{n-4}}{c^{n-4}} & \dots & 1 & \frac{(-1)^{n-2}(n-1)\sigma_{n-1}}{c^{n-1}}\\
\frac{(-1)^{n-1}\sigma_{n-1}}{c^{n-1}} & \frac{(-1)^{n-2}\sigma_{n-2}}{c^{n-2}} & \frac{(-1)^{n-3}\sigma_{n-3}}{c^{n-3}} & \dots & -\frac{\sigma_{1}}{c} & \frac{(-1)^{n-1}n\sigma_{n}}{c^{n}}\end{array}\right),\label{eq:73}\end{align}

for any $c\neq0$.
\end{thm}
\begin{proof}
Clearly, the condition \eqref{eq:66} implies that \eqref{eq:67}
and \eqref{eq:68} converge absolutely and uniformly on any compact
subset of $\mathbb{C}$. Thus, $f(z)$ is an entire function. Then
we have\begin{align}
(-1)f'(z) & =f(z)\sum_{n=1}^{\infty}\frac{\lambda_{n}}{1-z\lambda_{n}}.\label{eq:74}\end{align}
Assume that \eqref{eq:68} is true for some positive integer $n$,
or\begin{align}
 & (-1)^{n}f^{(n)}(z)=n!f(z)\label{eq:75}\\
 & \times\sum_{1\le k_{1}<k_{2}\dots<k_{n}}\frac{\lambda_{k_{1}}\lambda_{k_{2}}\dotsb\lambda_{k_{n}}}{(1-z\lambda_{k_{1}})(1-z\lambda_{k_{2}})\dotsb(1-z\lambda_{k_{n}})},\nonumber \end{align}
then,\begin{align}
 & (-1)^{n+1}f^{(n+1)}(z)=(-1)n!f'(z)\label{eq:76}\\
 & \times\sum_{1\le k_{1}<k_{2}\dots<k_{n}}\frac{\lambda_{k_{1}}\lambda_{k_{2}}\dotsb\lambda_{k_{n}}}{(1-z\lambda_{k_{1}})(1-z\lambda_{k_{2}})\dotsb(1-z\lambda_{k_{n}})}\nonumber \\
 & -n!f(z)\sum_{1\le k_{1}<k_{2}\dots<k_{n}}\frac{\lambda_{k_{1}}^{2}\lambda_{k_{2}}\dotsb\lambda_{k_{n}}}{(1-z\lambda_{k_{1}})^{2}(1-z\lambda_{k_{2}})\dotsb(1-z\lambda_{k_{n}})}\nonumber \\
 & -n!f(z)\sum_{1\le k_{1}<k_{2}\dots<k_{n}}\frac{\lambda_{k_{1}}\lambda_{k_{2}}^{2}\dotsb\lambda_{k_{n}}}{(1-z\lambda_{k_{1}})(1-z\lambda_{k_{2}})^{2}\dotsb(1-z\lambda_{k_{n}})}\nonumber \\
 & -\dots\nonumber \\
 & -n!f(z)\sum_{1\le k_{1}<k_{2}\dots<k_{n}}\frac{\lambda_{k_{1}}\lambda_{k_{2}}\dotsb\lambda_{k_{n}}^{2}}{(1-z\lambda_{k_{1}})(1-z\lambda_{k_{2}})\dotsb(1-z\lambda_{k_{n}})^{2}}\nonumber \\
 & =n!f(z)\left\{ \sum_{1\le k_{1}<k_{2}\dots<k_{n}}\sum_{k=1}^{\infty}\frac{\lambda_{k}}{1-z\lambda_{k}}\frac{\lambda_{k_{1}}\lambda_{k_{2}}\dotsb\lambda_{k_{n}}}{(1-z\lambda_{k_{1}})(1-z\lambda_{k_{2}})\dotsb(1-z\lambda_{k_{n}})}\right.\nonumber \\
 & -\sum_{1\le k_{1}<k_{2}\dots<k_{n}}\frac{\lambda_{k_{1}}^{2}\lambda_{k_{2}}\dotsb\lambda_{k_{n}}}{(1-z\lambda_{k_{1}})^{2}(1-z\lambda_{k_{2}})\dotsb(1-z\lambda_{k_{n}})}\nonumber \\
 & -\sum_{1\le k_{1}<k_{2}\dots<k_{n}}\frac{\lambda_{k_{1}}\lambda_{k_{2}}^{2}\dotsb\lambda_{k_{n}}}{(1-z\lambda_{k_{1}})(1-z\lambda_{k_{2}})^{2}\dotsb(1-z\lambda_{k_{n}})}\nonumber \\
 & -\dots\nonumber \\
 & \left.-\sum_{1\le k_{1}<k_{2}\dots<k_{n}}\frac{\lambda_{k_{1}}\lambda_{k_{2}}\dotsb\lambda_{k_{n}}^{2}}{(1-z\lambda_{k_{1}})(1-z\lambda_{k_{2}})\dotsb(1-z\lambda_{k_{n}})^{2}}\right\} .\nonumber \end{align}
The first sum within the braces could be split into several sums according
to whether $k$ equals one of these $k_{1},\dots,k_{n}$ or in one
of the $n+1$ intervals:\begin{equation}
(1,k_{1}),(k_{1},k_{2}),\dots(k_{n},\infty).\label{eq:77}\end{equation}
It is clear that the $n$ sums in the first case cancel out the last
$n$ negative sums within the braces, while each of the $n+1$ sums
obtained from the last case, after renaming the dummy variables, equals
\begin{align}
\sum_{1\le k_{1}<k_{2}\dots<k_{n+1}}\frac{\lambda_{k_{1}}\lambda_{k_{2}}\dotsb\lambda_{k_{n+1}}}{(1-z\lambda_{k_{1}})(1-z\lambda_{k_{2}})\dotsb(1-z\lambda_{k_{n+1}})},\label{eq:78}\end{align}
which implies that \eqref{eq:68} holds for $n+1$, and the proof
of \eqref{eq:67} is finished by the principle of induction. Observe
that\begin{align}
-f'(z) & =\sum_{n=0}^{\infty}(-1)^{n}(n+1)\sigma_{n+1}z^{n}\label{eq:79}\end{align}
 and\begin{align}
\sum_{n=1}^{\infty}\frac{\lambda_{n}}{1-z\lambda_{n}} & =\sum_{n=0}^{\infty}s_{n+1}z^{n},\label{eq:80}\end{align}
 equation \eqref{eq:74} becomes\begin{align}
\sum_{n=0}^{\infty}(-1)^{n}(n+1)\sigma_{n+1}z^{n} & =\left(\sum_{n=0}^{\infty}(-1)^{n}\sigma_{n}z^{n}\right)\left(\sum_{n=0}^{\infty}s_{n+1}z^{n}\right).\label{eq:81}\end{align}
We get \eqref{eq:72} by equating the corresponding coefficients of
$z^{n}$. \eqref{eq:73} is obtained from \eqref{eq:72} by solving
for $\left\{ \frac{s_{1}}{c},\frac{s_{2}}{c^{2}},\dots,\frac{s_{n}}{c^{n}}\right\} $
with Cramer's rule. 
\end{proof}

\section{Applications\label{sec:Applications}}

\subsection{Sine function \textmd{$\sin(z)$}}

In the Theorem \ref{thm:3.1} we take \begin{align}
\lambda_{k} & =\frac{1}{k^{2}},\quad k\in\mathbb{N},\label{eq:82}\end{align}
then,\begin{align}
f(z) & =\prod_{k=1}^{\infty}\left(1-\frac{z}{k^{2}}\right)=\frac{\sin\left(\pi\sqrt{z}\right)}{\pi\sqrt{z}}.\label{eq:83}\end{align}
Since\begin{align}
\frac{\sin\left(\pi\sqrt{z}\right)}{\pi\sqrt{z}} & =\sum_{n=0}^{\infty}\frac{\pi^{2k}}{(2k+1)!}(-z)^{k},\label{eq:84}\end{align}
then,\begin{align}
\sum_{1\le k_{1}<k_{2}<\dots<k_{n}}\frac{1}{k_{1}^{2}\cdot k_{2}^{2}\cdots k_{n}^{2}} & =\frac{\pi^{2n}}{(2n+1)!},\label{eq:85}\end{align}
which is known, see \cite{Borwein}. In this case $s_{k}$ has a very
nice formula, which was discovered by Euler \cite{Andrews2}: \begin{align}
\zeta(2k) & =\sum_{n=1}^{\infty}\frac{1}{n^{2k}}=\frac{(-1)^{k-1}2^{2k-1}B_{2k}\pi^{2k}}{(2k)!},\label{eq:86}\end{align}
 where $B_{2k}$ is the $2k$-th Bernoulli number defined by\begin{align}
\frac{x}{e^{x}-1} & =\sum_{n=0}^{\infty}B_{n}\frac{x^{n}}{n!}.\label{eq:87}\end{align}
Then, apply \eqref{eq:73} with $c=-\pi^{2}$ to get

\begin{align}
 & \frac{2^{2n-1}B_{2n}}{(2n)!}\label{eq:88}\\
= & \det\left(\begin{array}{cccccc}
1 & 0 & 0 & \dots & 0 & \frac{1}{3!}\\
\frac{1}{3!} & 1 & 0 & \dots & 0 & \frac{2}{5!}\\
\frac{1}{5!} & \frac{1}{3!} & 1 & \dots & 0 & \frac{3}{7!}\\
\vdots & \vdots & \vdots & \ddots & \vdots & \vdots\\
\frac{1}{(2n-3)!} & \frac{1}{(2n-5)!} & \frac{1}{(2n-7)!} & \dots & 1 & \frac{(n-1)}{(2n-1)!}\\
\frac{1}{(2n-1)!} & \frac{1}{(2n-3)!} & \frac{1}{(2n-5)!} & \dots & \frac{1}{3!} & \frac{n}{(2n+1)!}\end{array}\right),\nonumber \end{align}
 which is known.

\subsection{Bessel function $J_{\nu}(z)$}

Take \begin{align}
\lambda_{k} & =\frac{1}{j_{\nu,k}^{2}},\quad k\in\mathbb{N},\label{eq:89}\end{align}
in Theorem \ref{thm:3.1}, then,\begin{align}
f(z) & =\frac{2^{\nu}\Gamma(\nu+1)J_{\nu}(z^{1/2})}{z^{\nu/2}}=\prod_{n=1}^{\infty}\left(1-\frac{z}{j_{\nu,n}^{2}}\right)\label{eq:90}\end{align}
has the series expansion\begin{align}
f(z) & =\sum_{n=0}^{\infty}\frac{(-1)^{n}\Gamma(\nu+1)z^{n}}{n!2^{2n}\Gamma(\nu+n+1)}.\label{eq:91}\end{align}
Hence\begin{align}
\sum_{1\le k_{1}<k_{2}<\dots<k_{n}}\frac{1}{j_{\nu,k_{1}}^{2}\cdot j_{\nu,k_{2}}^{2}\cdots j_{\nu,k_{n}}^{2}} & =\frac{\Gamma(\nu+1)}{n!2^{2n}\Gamma(\nu+n+1)},\label{eq:92}\end{align}
or\begin{align}
\sum_{1\le k_{1}<k_{2}<\dots<k_{n}}\frac{1}{j_{\nu,k_{1}}^{2}\cdot j_{\nu,k_{2}}^{2}\cdots j_{\nu,k_{n}}^{2}} & =\frac{1}{n!4^{n}(\nu+1)_{n}}.\label{eq:93}\end{align}
Then, from \eqref{eq:73} with $c=-\frac{1}{4}$ to obtain\begin{align}
 & \sum_{k=1}^{\infty}\frac{4^{n}(-1)^{n-1}}{j_{\nu,k}^{2n}}\label{eq:94}\\
 & =\det\left(\begin{array}{cccccc}
1 & 0 & 0 & \dots & 0 & \frac{1}{(\nu+1)}\\
\frac{1}{(\nu+1)} & 1 & 0 & \dots & 0 & \frac{1}{(\nu+1)_{2}}\\
\frac{1}{2!(\nu+1)_{2}} & \frac{1}{(\nu+1)} & 1 & \dots & 0 & \frac{1}{2!(\nu+1)_{3}}\\
\vdots & \vdots & \vdots & \ddots & \vdots & \vdots\\
\frac{1}{(n-2)!(\nu+1)_{n-2}} & \frac{1}{(n-3)!(\nu+1)_{n-3}} & \frac{1}{(n-4)!(\nu+1)_{n-4}} & \dots & 1 & \frac{1}{(n-2)!(\nu+1)_{n-1}}\\
\frac{1}{(n-1)!(\nu+1)_{n-1}} & \frac{1}{(n-2)!(\nu+1)_{n-2}} & \frac{1}{(n-3)!(\nu+1)_{n-3}} & \dots & \frac{1}{(\nu+1)} & \frac{1}{(n-1)!(\nu+1)_{n}}\end{array}\right).\nonumber \end{align}
Here are the first few $s_{n}$s, \cite{Watson},

\begin{align}
s_{1} & =\frac{1}{4(\nu+1)_{1}},\label{eq:95}\\
s_{2} & =\frac{1}{4^{2}\prod_{j=1}^{2}(\nu+1)_{j}},\label{eq:96}\\
s_{3} & =\frac{2(\nu+2)_{1}}{4^{3}\prod_{j=1}^{3}(\nu+1)_{j}},\label{eq:97}\\
s_{4} & =\frac{(5\nu+11)(v+2)_{2}}{4^{4}\prod_{j=1}^{4}(\nu+1)_{j}},\label{eq:98}\\
s_{5} & =\frac{2(7\nu+19)(\nu+2)_{2}(\nu+2)_{3}}{4^{5}\prod_{j=1}^{5}(\nu+1)_{j}}.\label{eq:99}\end{align}

\subsection{Airy function $A(z)$}

Take\begin{align}
\lambda_{k} & =\frac{1}{i_{k}^{2}},\quad k\in\mathbb{N}.\label{eq:100}\end{align}
From \eqref{eq:21} and \eqref{eq:22} to obtain \begin{align}
\sum_{1\le k_{1}<k_{2}<\dots<k_{n}}\frac{1}{i_{k_{1}}^{2}\cdot i_{k_{2}}^{2}\cdots i_{k_{n}}^{2}} & =\frac{\sqrt{3}\Gamma\left(\frac{2}{3}\right)^{2}}{\sqrt[3]{4}\pi}\frac{16^{n/3}\Gamma\left(\frac{n}{3}+\frac{1}{6}\right)\Gamma\left(\frac{n}{3}+\frac{1}{2}\right)}{(2n)!}.\label{eq:101}\end{align}
Let\begin{equation}
a(n)=\frac{\sqrt{3}\Gamma\left(\frac{2}{3}\right)^{2}}{\sqrt[3]{4}\pi}\frac{\Gamma\left(\frac{n}{3}+\frac{1}{6}\right)\Gamma\left(\frac{n}{3}+\frac{1}{2}\right)}{(2n)!},\label{eq:102}\end{equation}
then,\begin{align}
 & \sum_{k=1}^{\infty}\frac{(-1)^{n-1}}{i_{k}^{2n}2^{4n/3}}\label{eq:103}\\
= & \det\left(\begin{array}{cccccc}
1 & 0 & 0 & \dots & 0 & a(1)\\
a(1) & 1 & 0 & \dots & 0 & 2a(2)\\
a(2) & a(1) & 1 & \dots & 0 & 3a(3)\\
\vdots & \vdots & \vdots & \ddots & \vdots & \vdots\\
a(n-2) & a(n-3) & a(n-4) & \dots & 1 & (n-1)a(n-1)\\
a(n-1) & a(n-2) & a(n-3) & \dots & -a(1) & na(n)\end{array}\right).\nonumber \end{align}
The first five $s_{n}$ are:\begin{align}
s_{1} & =\frac{\sqrt{3}\Gamma\left(\frac{2}{3}\right)^{2}\Gamma\left(\frac{5}{6}\right)}{2^{1/3}\sqrt{\pi}}\label{eq:104}\\
s_{2} & =\frac{3\Gamma\left(\frac{2}{3}\right)^{4}\Gamma\left(\frac{5}{6}\right)^{2}}{2^{2/3}\pi}-\frac{\Gamma\left(\frac{2}{3}\right)\Gamma\left(\frac{5}{3}\right)}{2\sqrt{3}},\label{eq:105}\\
s_{3} & =\frac{\pi}{90}+\frac{3\sqrt{3}\Gamma\left(\frac{2}{3}\right)^{6}\Gamma\left(\frac{5}{6}\right)^{3}}{2\pi^{3/2}}-\frac{3\Gamma\left(\frac{2}{3}\right)^{3}\Gamma\left(\frac{5}{6}\right)\Gamma\left(\frac{5}{3}\right)}{42^{1/3}\sqrt{\pi}},\label{eq:106}\\
s_{4} & =\frac{1}{54}\Gamma\left(\frac{2}{3}\right)^{4}+\frac{(56\pi-5)\Gamma\left(\frac{2}{3}\right)^{2}\Gamma\left(\frac{5}{6}\right)}{12602^{1/3}\sqrt{3\pi}}\label{eq:107}\\
+ & \frac{92^{2/3}\Gamma\left(\frac{2}{3}\right)^{8}\Gamma\left(\frac{5}{6}\right)^{4}-22^{1/3}\sqrt{3}\pi\Gamma\left(\frac{2}{3}\right)^{5}\Gamma\left(\frac{5}{6}\right)^{2}\Gamma\left(\frac{5}{3}\right)}{4\pi^{2}},\nonumber \\
s_{5} & =\frac{5\Gamma\left(\frac{2}{3}\right)^{6}\Gamma\left(\frac{5}{6}\right)}{362^{1/3}\sqrt{3\pi}}+\frac{(56\pi-5)\Gamma\left(\frac{2}{3}\right)^{4}\Gamma\left(\frac{5}{6}\right)^{2}}{10082^{2/3}\pi}\label{eq:108}\\
- & \frac{5\Gamma\left(\frac{2}{3}\right)^{8}\Gamma\left(\frac{5}{6}\right)^{3}}{4\pi^{3/2}}+\frac{9\sqrt{3}\Gamma\left(\frac{2}{3}\right)^{10}\Gamma\left(\frac{5}{6}\right)^{5}}{22^{2/3}\pi^{5/2}}+\frac{(1-36\pi)\Gamma\left(\frac{2}{3}\right)\Gamma\left(\frac{8}{3}\right)}{12960\sqrt{3}}.\nonumber \end{align}

\subsection{$q$-Bessel function $J_{\nu}^{(2)}(z;q)$}

Take\begin{align}
\lambda_{k} & =\frac{1}{j_{\nu,k}^{2}(q)},\quad k\in\mathbb{N}.\label{eq:109}\end{align}
From\begin{align}
 & \frac{2^{\nu}(q;q)_{\infty}J_{\nu}^{(2)}(z^{1/2};q)}{(q^{\nu+1};q)_{\infty}z^{\nu/2}}=\prod_{n=1}^{\infty}\left(1-\frac{z}{j_{\nu,n}^{2}}\right)\label{eq:110}\\
 & =\sum_{n=0}^{\infty}\frac{(-1)^{n}q^{n(n+\nu)}z^{n}}{(q,q^{\nu+1};q)_{n}2^{2n}},\nonumber \end{align}
to obtain\begin{align}
\sum_{1\le k_{1}<k_{2}<\dots<k_{n}}\frac{1}{j_{\nu,k_{1}}^{2}(q)\cdot j_{\nu,k_{2}}^{2}(q)\cdots j_{\nu,k_{n}}^{2}(q)} & =\frac{q^{n(n+\nu)}}{4^{n}(q,q^{\nu+1};q)_{n}}.\label{eq:111}\end{align}
and\begin{align}
 & \sum_{k=1}^{\infty}\frac{(-1)^{n-1}4^{n}}{q^{n\nu}j_{\nu,k}^{2n}(q)}\label{eq:112}\\
= & \det\left(\begin{array}{cccccc}
1 & 0 & 0 & \dots & 0 & b(1;q)\\
b(1;q) & 1 & 0 & \dots & 0 & 2b(2;q)\\
b(2;q) & b(1;q) & 1 & \dots & 0 & 3b(3;q)\\
\vdots & \vdots & \vdots & \ddots & \vdots & \vdots\\
b(n-2;q) & b(n-3;q) & b(n-4;q) & \dots & 1 & (n-1)b(n-1;q)\\
b(n-1;q) & b(n-2;q) & b(n-3;q) & \dots & b(1;q) & nb(n;q)\end{array}\right),\nonumber \end{align}
 where\begin{align}
b(n;q): & =\frac{q^{n^{2}}}{(q,q^{\nu+1};q)_{n}}.\label{eq:113}\end{align}
The first three $s_{n}$s are:

\begin{align}
s_{1} & =\frac{q^{\nu+1}}{4(1-q)(1-q^{\nu+1})},\label{eq:114}\\
s_{2} & =\frac{q^{2(\nu+1)}(1+2q-q^{\nu+2})}{4^{2}(1-q^{2})(1-q^{\nu+1})(q^{\nu+1};q)_{2}},\label{eq:115}\\
s_{3} & =\frac{q^{3(\nu+1)}\left(1+3q+3q^{2}+3q^{3}-q^{\nu+2}-q^{\nu+3}-3q^{\nu+4}+q^{2\nu+5}\right)}{4^{3}(1-q^{3})\left(1-q^{\nu+1}\right)^{2}(q^{\nu+1};q)_{3}}.\label{eq:116}\end{align}

\subsection{Ramanujan's entire function $A_{q}(z)$}

Take\begin{equation}
\lambda_{k}=\frac{1}{i_{k}(q)},\quad k\in\mathbb{N}.\label{eq:117}\end{equation}
From\begin{align}
A_{q}(z) & =\sum_{n=0}^{\infty}\frac{q^{n^{2}}(-1)^{n}z^{n}}{(q;q)_{n}}=\prod_{n=1}^{\infty}\left(1-\frac{z}{i_{n}(q)}\right)\label{eq:118}\end{align}
to get\begin{align}
\sum_{1\le k_{1}<k_{2}<\dots<k_{n}}\frac{1}{i_{k_{1}}(q)\cdot i_{k_{2}}(q)\cdots i_{k_{n}}(q)} & =\frac{q^{n^{2}}}{(q;q)_{n}},\label{eq:119}\end{align}
 and\begin{align}
 & \sum_{k=1}^{\infty}\frac{(-1)^{n-1}}{i_{k}^{n}(q)}\label{eq:120}\\
= & \det\left(\begin{array}{cccccc}
1 & 0 & 0 & \dots & 0 & \frac{q}{1-q}\\
\frac{q}{1-q} & 1 & 0 & \dots & 0 & \frac{2q^{4}}{(q;q)_{2}}\\
\frac{q}{1-q} & \frac{q}{1-q} & 1 & \dots & 0 & \frac{3q^{9}}{(q;q)_{3}}\\
\vdots & \vdots & \vdots & \ddots & \vdots & \vdots\\
\frac{q^{(n-2)^{2}}}{(q;q)_{n-2}} & \frac{q^{(n-3)^{2}}}{(q;q)_{n-3}} & \frac{q^{(n-4)^{2}}}{(q;q)_{n-4}} & \dots & 1 & \frac{(n-1)q^{(n-1)^{2}}}{(q;q)_{n-1}}\\
\frac{q^{(n-1)^{2}}}{(q;q)_{n-1}} & \frac{q^{(n-2)^{2}}}{(q;q)_{n-2}} & \frac{q^{(n-3)^{2}}}{(q;q)_{n-3}} & \dots & \frac{q}{1-q} & \frac{nq^{n^{2}}}{(q;q)_{n}}\end{array}\right).\nonumber \end{align}
The first five $s_{n}$s are:

\begin{align}
s_{1} & =\frac{q}{1-q},\label{eq:121}\\
s_{2} & =\frac{q^{2}(1+2q)}{1-q^{2}},\label{eq:122}\\
s_{3} & =\frac{q^{3}(1+3q+3q^{2}+3q^{3})}{1-q^{3}},\label{eq:123}\\
s_{4} & =\frac{q^{4}(1+2q+2q^{3})(1+2q+2q^{2}+2q^{3})}{1-q^{4}},\label{eq:124}\\
s_{5} & =\frac{q^{5}}{1-q^{5}}\label{eq:125}\\
\times & (1+5q+10q^{2}+15q^{3}+20q^{4}+20q^{5}+20q^{6}+15q^{7}+10q^{8}+5q^{9}+5q^{10}).\nonumber \end{align}

\subsection{Riemann zeta function $\zeta(s)$}

Take \begin{equation}
\lambda_{k}=\frac{1}{z_{k}^{2}},\quad k\in\mathbb{N}.\label{eq:126}\end{equation}
From \eqref{eq:46} to obtain\begin{align}
\sum_{1\le k_{1}<k_{2}<\dots<k_{n}}\frac{1}{z_{k_{1}}^{2}\cdot z_{k_{2}}^{2}\cdots z_{k_{n}}^{2}}=\frac{b_{n}}{(2n)!b_{0}} & ,\label{eq:127}\end{align}
 and\begin{align}
 & \sum_{k=1}^{\infty}\frac{(-1)^{n-1}b_{0}^{n}}{z_{k}^{2n}}\label{eq:128}\\
= & \det\left(\begin{array}{cccccc}
b_{0} & 0 & 0 & \dots & 0 & \frac{b_{1}}{2!}\\
\frac{b_{1}}{2!} & 1 & 0 & \dots & 0 & \frac{2b_{2}}{4!}\\
\frac{b_{2}}{4!} & \frac{b_{1}}{2!} & 1 & \dots & 0 & \frac{3b_{3}}{6!}\\
\vdots & \vdots & \vdots & \ddots & \vdots & \vdots\\
\frac{b_{n-2}}{(2n-4)!} & \frac{b_{n-3}}{(2n-6)!} & \frac{b_{n-4}}{(2n-8)!} & \dots & 1 & \frac{(n-1)b_{n-1}}{(2n-2)!}\\
\frac{b_{n-1}}{(2n-2)!} & \frac{b_{n-2}}{(2n-4)!} & \frac{b_{n-3}}{(2n-6)!} & \dots & \frac{b_{1}}{2!} & \frac{nb_{n}}{(2n)!}\end{array}\right).\nonumber \end{align}
The first four $s_{n}$s are:

\begin{align}
s_{1} & =\frac{b_{1}}{2b_{0}},\label{eq:129}\\
s_{2} & =\frac{3b_{1}^{2}-b_{0}b_{2}}{12b_{0}^{2}},\label{eq:130}\\
s_{3} & =\frac{30b_{1}^{3}-15b_{0}b_{1}b_{2}+b_{0}^{2}b_{3}}{240b_{0}^{3}},\label{eq:131}\\
s_{4} & =\frac{630b_{1}^{4}-420b_{0}b_{1}^{2}b_{2}+35b_{0}^{2}b_{2}^{2}+28b_{0}^{2}b_{1}b_{3}-b_{0}^{3}b_{4}}{10080b_{0}^{4}}.\label{eq:132}\end{align}
The Dirichlet $L$ series have similar formulas with $z_{k}$ being
replaced by $z_{k}(\chi,a)$, and $b_{k}$ by $b_{k}(\chi,a)$.

\begin{acknowledgement*}
This work is partially supported by Chinese National Natural Science
Foundation grant No.10761002, Guangxi Natural Science Foundation grant
No.0728090.
\end{acknowledgement*}

\end{document}